\documentclass[11pt]{article}
\DeclareMathAlphabet{\mathpzc}{OT1}{pzc}{m}{it}
\usepackage{amssymb,amsmath,amsthm}
\usepackage{extarrows}
\usepackage{enumerate,multicol}
\usepackage[capposition=below]{floatrow}
\usepackage{fullpage}
\usepackage{graphicx}
\usepackage[mathcal]{euscript}
\usepackage{mathrsfs}
\usepackage{upgreek}
\usepackage{MnSymbol}
\usepackage{marginnote}
\usepackage{tikz,tikz-cd}
\usepackage{tikz-3dplot}
\usepackage{comment}
\usepackage{pgfplots}
\usepackage{subfig}
\usepackage{hyperref}
\usepackage{musicography}
\usepackage{breakcites}

\pgfplotsset{compat=1.15}
\usepgfplotslibrary{fillbetween}
\usetikzlibrary{matrix,arrows}
\theoremstyle{plain}
\numberwithin{equation}{section}
\newtheorem{thm}{Theorem}[section]

\newtheorem{cor}[thm]{Corollary}

\theoremstyle{definition}

\newtheorem{lemma}[thm]{Lemma}
\newtheorem{rmk}[thm]{Remark}
\newtheorem{defin}[thm]{Definition}

\newcommand\lip{\textup{lip}}
\newcommand\hol{\textup{hol}}
\newcommand\loc{\textup{loc}}

\newcommand{\suchthat}{\;\ifnum\currentgrouptype=16 \middle\fi|\;}
\DeclareMathAlphabet\mathbfcal{OMS}{cmsy}{b}{n}
\tikzset{
  subseteq/.style={
    draw=none,
    edge node={node [sloped, allow upside down, auto=false]{$\subseteq$}}}}
\hypersetup{colorlinks,
            allcolors = blue} 
 
 \renewcommand\qedsymbol{$\blacksquare$}   
\date{}

\begin{document}

\title{Topologies for geometric flows and continuous dependence on parameters}
\author{Andrew D. Lewis \footnote{Professor, Department of Mathematics and Statistics, Queen’s University, Kingston, ON K7L 3N6, Canada.} \and Yanlei Zhang\footnote{Postdoc, Mila - Quebec AI Institute,
Montreal, Quebec H2S 3H1, Canada \newline \hspace*{1.8em}Email:\href{mailto:yanlei.zhang@queensu.ca}{yanlei.zhang@mila.quebec}}}
\maketitle

\begin{abstract}
We study time- and parameter-dependent ordinary differential equations in the geometric setting of vector fields and their flows. Various degrees of regularities in state are considered, including Lipschitz, finitely diferentiable, smooth, and holomorphic. A suitable topology for the space of flows is derived using geometric descriptions of suitable topologies for vector fields. A new kind of continuous dependence is proved, that of the fixed time local flow on the parameter in a general topological space.
\vspace{5pt}


\vspace{5pt}
\end{abstract}

\section{Introduction}
This paper is concerned with a classical subject: time- and parameter-dependent ordinary differential equations, albeit in the geometric setting of vector fields and with a degree of generality one does not find with classical results. The objectives of the paper are to:
\begin{enumerate}
    \item provide a suitable topology for the space of parameter-dependent flows with a broad range of regularity in state;
    \item prove a very general continuity result for the ``parameter to local flow'' mapping.
\end{enumerate}
Before we give an overview of our approach and results, we shall review the commonly accepted state of play, as this is indeed well trod ground.

To allow for a discussion with some context, for the moment we consider an initial value problem on a smooth manifold $M$,
$$\xi'(t)=F(t,\xi(t),p),\ \ \xi(t_0)=x_0,$$
with a solution $t\to \xi(t)\in M$ and $p$ a parameter in a topological space $\mathcal{P}$. We write the solution of the initial value problem as  
$$(t,t_0,x_0,p)\to \Phi^F(t,t_0,x_0,p),\ \ p\in\mathcal{P},$$
to include all of its dependencies. This is typically thought of as the image of the family of time-dependent vector fields $(t,x)\to F_p(t,x):=F(t,x,p)$ under the map
\begin{eqnarray}\label{exp}
    \text{exp}:\{\text{vector fields}\}&\rightarrow& \{\text{local flows}\}\\
	  F&\mapsto& \Phi^F. \nonumber
\end{eqnarray}
This map plays a key role in understanding the basic structural attributes of control systems in control theory literature (\cite{Agrachev(2004),Jurdjevic(1997), BulloandLewis(2005)}).

The main contribution of this paper is concerned with two attributes. 
\begin{enumerate}
    \item We consider the regularities of vector fields and flows. The discussion of this matter can be broken down into three attributes of vector fields: time, state and parameter. 
    \begin{enumerate}
        \item For time-independent vector fields, the basic regularity results are well-established for all the standard degrees of regularity (\cite{Coddington(1955)}), including real analyticity (\cite{Sontag(1998)}). In the time-varying case, the standard Carath\'eodory existence and uniqueness theorem for ordinary differential equations depending measurably on time is a part of classical treatments (\cite{Coddington(1955)}). Continuity with respect to initial condition is proved by \cite{Sontag(1998)}.
        \item  Regarding state, in the situation where $x\to F(t,x)$ has some regularity, the matter of when this regularity is shared by $x \to \Phi^F(t, t_0, x)$ is dealt with more sparsely in the ordinary differential equation literature. For example, this arises in the chronological calculus approach of \cite{Agracev1979}, an approach which is given a nice outline in \cite{Agrachev(2004)}. The standard assumption for dependence on state that allows for the proof of the existence and uniqueness theorem is Lipschitz regularity. If the vector fields depends on state in a more regular way, one has standard theorems giving correspondingly more regular dependence of the flow on initial condition \cite{Hartman(1964)}. 
        \item Dependence on parameters in the general setting of measurable time-dependence is required in applications to optimal control. For example, \cite{Sontag(1998)} proves a quite general theorem for continuous dependence of the final state on initial state $x_0$ in the presence of measurable time-dependence. The case of differentiable dependence on parameters is considered in \cite{Hestenes(1966)} and \cite{Klose(2011)} by requring the parameter space to be metrisable, since a crucial role is played by the Dominated Convergence Theorem, which necessitates that the topology of the parameter space be describable by sequences. \cite{Schuricht2000} carefully consider differentiable dependence on state with continuous dependence on parameters. These authors also point out the paucity of results in these directions concerning differential equations with measurable time dependence. We do not consider here the problem of differentiable dependence on parameters, as our immediate concerns are for making the parameter spaces as general as possible, e.g., not necessarily being able to support a theory of differentiation. 
    \end{enumerate}
    \item We consider the suitable topologies for the space of vector fields and their flows.  This general idea is formulated by \cite{Jafarpourandlewis(2014)} using seminorms defined intrinsically in terms of fibre norms on jet bundles to include a variety of regularity classes, e.g., Lipschitz, finitely differentiable, smooth, holomorphic, and real analytic. We will follow the same treatment as \cite{Jafarpourandlewis(2014)} for the topologies of vector fields with different regularities. In Section \ref{section2}, we will give an explicit description for the topologies on spaces of time- and parameter-dependent vector fields, and derive a new topology for the space of flows. As a consequence, a general continuous dependence result is proved in Section \ref{section4}, that of the fixed time local flow on a parameter in a general topological space. 
\end{enumerate}

\section{Topologies}\label{section2} 
We shall give a brief outline of the notation we use in the paper. We shall mainly give definitions, establish the bare minimum of facts we require, and refer the reader to the references for details.

We shall assume all manifolds to be Hausdorff, second countable, and connected. We shall work with manifolds and vector bundles coming from different categories: smooth (i.e., infinitely differentiable), real analytic, and holomorphic (i.e., complex analytic). We shall use ``class $C^r$" to denote these three cases, i.e., $r\in\{\infty,\omega,\text{hol}\}$ for smooth, real analytic, and holomorphic, respectively.

Let $m\in\mathbb{Z}_{\geq 0}$ and let $m'\in \{0,\text{lip}\}$. We shall work with objects with regularity $\nu\in \{m+m',\infty,\omega,\text{hol}\}$. Thus $\nu=m$ means ``$m$-times continuously differentiable," $\nu=m+\text{lip}$ means ``$m$-times continuously differentiable with locally Lipschitz top derivative," $\nu=\infty$ means ``smooth", $\nu=\omega$ means ``real analytic", and $\nu=\text{hol}$ means ``holomorphic." Given $\nu\in \{m+m',\infty,\omega,\text{hol}\}$, we shall say ``let $r\in\{\infty,\omega,\text{hol}\}$, as required." This has the obvious meaning that $r=\text{hol}$ when $\nu=\text{hol}$, $r=\omega$ when $\nu=\omega$, and $r=\infty$ otherwise. We shall also use the terminology ``let $\mathbb{F}\in\{\mathbb{R},\mathbb{C}\}$, as appropriate." This means that $\mathbb{F}=\mathbb{C}$ when $r=\text{hol}$ and $\mathbb{F}=\mathbb{R}$ otherwise. 

Let $M, N$ be $C^r$-manifolds and let $\pi: E\to M$ be a $C^r$-vector bundle. We shall denote by $\Gamma^\nu(E)$, $C^\nu(M)$ and $C^\nu(M;N)$ the spaces of sections, functions and mappings of regularity $\nu$, respectively.

\subsection{Topologies for vector fields}
We will provide explicit seminorms that define the various topologies we use for local sections, corresponding to regularity classes $\nu\in \{m+m',\infty,\omega,\text{hol}\}$. We shall not use much space to describe the nature of these topologies, but give a general sketch and refer the interested readers to  \cite{Jafarpourandlewis(2014)} for more details.

As we are interested in ordinary differential equations with well-defined flows, we must, according to the usual theory, consider locally Lipschitz sections of vector bundles. In particular, we will find it essential to topologise the space of locally Lipschitz sections of $\pi : E \to M$. To define the seminorms for this topology, we make use of a ``local least Lipschitz constant."

We let $\xi : M \to E$ be such that $\xi(x) \in E_x$ for every $x \in M$. For a piecewise differentiable curve $\gamma : [0, T] \to M$, we denote by $\tau_{\gamma,t} : E_{\gamma(0)} \to E_{\gamma(t)}$ the isomorphism of parallel translation along $\gamma$ for each $t \in [0, T]$. We then define, for $K \subseteq M$ compact,
\begin{multline*}
    l_K(\xi)=\sup\bigg\{\frac{||\tau^{-1}_{\gamma,1}(\xi\circ\gamma(1))-\xi\circ\gamma(0)||_{\mathbb{G}_\pi}}{\ell_{\mathbb{G}_M}(\gamma)}\bigg|\gamma:[0,1]\to M,\  \gamma(0), \gamma(1) \in K,\\
    \gamma(0) \neq \gamma(1)\bigg\},
\end{multline*}
which is the \textbf{K-sectional dilatation of} $\xi$. Here $\ell_{\mathbb{G}_M}$ is the length function on piecewise differentiable curves. We also define
\begin{eqnarray*}
  \text{dil}\ \xi: M &\to& \mathbb{R}_{\geq 0}\\
  x&\mapsto&\inf\{l_{\text{cl}}(\mathcal{U})(\xi) \ |\  \mathcal{U} \text{ is a relatively compact neighbourhood of } x\},
\end{eqnarray*}
which is the \textbf{local sectional dilatation} of $\xi$. Note that, while the values taken by $\text{dil}\ \xi$ will depend on the choice of a Riemannian metric $\mathbb{G}$, the property $\text{dil}\  \xi(x) < \infty$ for $x \in M$ is independent of $\mathbb{G}$. Moreover, $\xi\in \Gamma^{\lip}(E)$ if and only if $\text{dil}\  \xi(x) < \infty$ for all $x \in M$. We refer the reader to \cite{Jafarpourandlewis(2014)} for more details.
\subsubsection{Fibre norms for jet bundles}
Fibre norms for jet bundles of a vector bundle play an important role in our unified treatment of various classes of regularities. Our discussion begins with general constructions for the fibres of jet bundles. Let $r\in \{\infty,\omega\}$ and let $M$ be a $C^r$-manifold.  Let $\pi:E\rightarrow M$ be a $C^r$-vector bundle with $\pi_m:J^mE\rightarrow M$ its $m$th jet bundle. We shall suppose that we have a $C^r$-affine connection $\nabla^M$ on $M$ and a $C^r$-vector bundle connection $\nabla^\pi$ in $E$. By additionally supposing that we have a $C^r$-Riemannian metric $\mathbb{G}_M$ on M and a $C^r$-fibre metric $\mathbb{G}_\pi$ on $E$, we shall give a $C^r$-fibre norm on $J^mE$.  We will not require $\nabla^M$ to be the Levi-Civita connection for the Riemannian metric $\mathbb{G}_M$, nor do we typically require there to be any metric relationship between $\nabla^\pi$ and $\mathbb{G}_\pi$. However, in the Lipschitz topology, it is sometimes convenient to assume that $\nabla^M$ is the Levi-Civita connection for $\mathbb{G}_M$ and that $\nabla^\pi$ is $\mathbb{G}_\pi$-orthogonal, i.e., parallel transport consists of inner product preserving mappings.

Denote by $T^m(T^*M)$ the $m$-fold tensor product of $T^*M$ and $S^m(T^*M)$ the symmetric tensor bundle. The connection $\nabla^M$ induces a covariant derivative for tensor fields $A\in \Gamma^1(T^k_l(TM))$ on $M$, $k,l\in\mathbb{Z}_{\geq 0}$. This covariant derivative we denote by $\nabla^M$,  dropping reference to the particular $k$ and $l$.  Similarly, the connection $\nabla^\pi$ induces a covariant derivative for sections $B\in \Gamma^1(T^k_l(E))$  of the tensor bundles associated with $E$, $k,l\in\mathbb{Z}_{\geq 0}$. This covariant derivative we denote by $\nabla^\pi$,  dropping reference to the particular $k$ and $l$.
We will also consider differentiation of sections of $T^{k_1}_{l_1}(TM) \otimes T^{k_2}_{l_2}(E)$, and we denote the covariant derivative by $\nabla^{M,\pi}$. Note that
$$\nabla^{M,\pi,m}\xi:=\underbrace{\nabla^{M,\pi}\cdots(\nabla^{M,\pi}}_{m-1\ \text{times}}(\nabla^\pi\xi))\in\Gamma^\infty(T^m(T^*M)\otimes E).$$
For $\xi\in \Gamma^\infty(E)$ and $m\in\mathbb{Z}_{\geq 0}$, we define
$$D^m_{\nabla^M,\nabla^\pi}(\xi)=\text{Sym}_m\otimes\text{id}_E(\nabla^{M,\pi,m}\xi)\in\Gamma^\infty(S^m(T^*M)\otimes E),$$
where $\text{Sym}_m:T^m(T^*M)\rightarrow S^m(T^*M)$ is given by
$$\text{Sym}_m(v_1\otimes\cdots\otimes v_m)=\frac{1}{m!}\sum_{\sigma\in\mathfrak{S}_m}v_{\sigma(1)}\otimes\cdots\otimes v_{\sigma(m)}.$$
We take the convention that $D^0_{\nabla^M,\nabla^\pi}(\xi)=\xi$. We then have a map
\begin{eqnarray*}
      S^m_{\nabla^M,\nabla^{M,\pi}}:J^mE &\rightarrow&  \bigoplus\limits_{j=0}^{m}(S^j(T^*M)\otimes E) \\
	         j_m\xi(x)&\mapsto& (\xi(x),D^1_{\nabla^M,\nabla^\pi}(\xi)(x),...,D^m_{\nabla^M,\nabla^\pi}(\xi)(x))
\end{eqnarray*}
that can be verified to be an  isomorphism of vector bundles \cite[Lemma 2.1]{Jafarpourandlewis(2014)}. Note that inner products on the components of a tensor products induce an inner product on the tensor product in a
natural way \cite[Lemma 2.3]{Jafarpourandlewis(2014)}. Then we have a fibre metric in all tensor spaces associated with $TM$ and $E$ and their tensor products. We shall denote by $G_{M,\pi}$ any of these various fibre metrics. In particular, we have a fibre metric $\mathbb{G}_{M,\pi}$ on $T^j(T^*M)\otimes E$ for each $j\in \mathbb{Z}_{\geq 0}$. This thus gives us a fibre metric $\mathbb{G}_{M,\pi,m}$ on $J^mE$ defined by
\begin{equation}\label{eq:fibremetric}
    \mathbb{G}_{M,\pi,m}(j_m\xi(x),j_m\eta(x))=\sum_{j=0}^{m}\mathbb{G}_{M,\pi}\left(\frac{1}{j!}D^j_{\nabla^M,\nabla^\pi}(\xi)(x),\frac{1}{j!}D^j_{\nabla^M,\nabla^\pi}(\eta)(x)\right).
\end{equation}
Associated to this inner product on fibres is the norm on fibres, which we denote by $\|\cdot\|_{\mathbb{G}_{M,\pi,m}}$. We shall use these fibre norms continually in our descriptions of various topologies in the next section. We comment that the factorials are required to make sense of the real analytic topology (\cite{Jafarpourandlewis(2014)}).

\subsubsection{Topoligies for space of sections}
We can now quickly describe the topologies for $\Gamma^{\nu}(E)$ that we use in the paper. We refer to
\cite{Jafarpourandlewis(2014)} for details.
Let $r \in \{\infty, \omega, \hol\}$ and let $\pi: E \to M$ be a $C^r$-vector bundle, let $m \in \mathbb{Z}_{\geq 0}$ and let $m' \in \{0, \lip\}$, let $\nu \in \{m+m', \infty, \omega, \hol\}$. Let $G_\pi$ be an Hermitian metric on the vector bundle and denote by $\|\cdot\|_{G_\pi}$ the associated fibre norm. As we explained in the previous section, we will assume $\nabla^M$ to be the Levi-Civita connection for the metric $\mathbb{G}_M$ in the Lipschitz case. We denote by $c_0(\mathbb{Z}_{\geq 0}; \mathbb{R}_{>0})$ the positive sequences in $\mathbb{R}$ converging to $0$. We define seminorms for $\Gamma^{\nu}(E)$ as follows:
\begin{enumerate}
    \item $\nu=m$: for a compact $K\subseteq M$, denote 
    $$p^m_K(\xi)=\sup\{\|j_m\xi(x)\|_{G_{M,\pi,m}}\ |\ x\in K\};$$
    \item $\nu=m+\lip$: for a compact $K\subseteq M$, denote 
    $$\lambda^m_K(\xi)=\sup\{\text{dil}j_m\xi(x)\ |\ x\in K\}, \hspace{10pt}  p^{m+\text{lip}}_K(\xi)=\max\{\lambda^m_K(\xi), p^m_K(\xi)\};$$
    \item $\nu=\infty$: for a compact $K\subseteq M$ and $m\in \mathbb{Z}_{\geq 0}$, denoted 
    $$p^\infty_{K,m}(\xi)=\sup\{\|j_m\xi(x)\|_{\mathbb{G}_{M,\pi,m}}\ |\ x\in K\};$$
    \item $\nu=\omega$: for a compact $K\subseteq M$ and $\boldsymbol{a}\in c_0(\mathbb{Z}_{\geq 0};\mathbb{R}_{>0})$, denote 
    $$p^\omega_{K,\boldsymbol{a}}(\xi)=\sup\{a_0a_1...a_m \|j_m\xi(x)\|_{\mathbb{G}_{M,\pi,m}}\ |\ x\in K,\ m\in \mathbb{Z}_{\geq 0}\};$$
    \item $\nu=\hol$: for a compact $K\subseteq M$, denote 
    $$p^{\text{hol}}_K(\xi)=\sup\{\|\xi(z)\|_{\mathbb{G}_\pi}\ |\ z\in K\}.$$
\end{enumerate}
Because these seminorms have a similar character, we shall often denote, for a compact $K \subseteq M$, the seminorm for $\Gamma^\nu(E)$ by $p^{\nu}_{K}$, with the specific ornamentation associated with a specific $\nu$ suppressed. This will allow us to treat all regularity classes simultaneously when it is convenient to do so.

We comment that these seminorms make it clear that we have an ordering of the regularity classes as
$$m_1<m_1+\lip<m_2<m_2+\lip<\cdots<\cdots<\infty<\omega<\hol$$
from least regular (coarser topology) to more regular (finer topology), and where $m_1 < m_2$. These seminorms define a locally convex topology for $\Gamma^\nu(E)$ which we simply call the $\mathbf{C^\nu}$\textbf{-topology}. We note that $p^\nu_K$ is Hausdorff, seperable, complete, webbed and ultrabornological (thus metrisable) for all regularities except $\nu=\omega$, therefore, the Open Mapping Theorem holds (\cite{Adasch(1978)}).

\subsubsection{Topologies for space of mappings}
Let $r \in \{\infty,\omega,\hol\}$ and let $M$ and $N$ be $C^r$-manifolds. For $m \in \mathbb{Z}_{\geq 0}$, $m' \in \{0, \lip\}$, and $ \nu \in \{m + m', \infty, \omega, \hol\}$, we have the set $C^\nu(M;N)$ of $C^\nu$-mappings. This spaces can be equipped with the initial topology of the mapping 
\begin{eqnarray*}
\Psi_f:C^\nu(M;N) &\rightarrow&  C^\nu(M)\\
	\Phi&\mapsto& \Phi^*f
\end{eqnarray*}
for $f\in C^r(N)$. This is equivalent to the topology of uniform convergence of derivatives on compact subsets and can be characterized by the semimetrics 
$$d^\nu_{K,f}(\Phi_1, \Phi_2)= \sup\{p^\nu_K(f\circ\Phi_1(x)-f\circ \Phi_2(x))\, | \, x\in K\},$$
where $\,\,\, f\in C^r(N), \, K\subseteq M$ compact, and $p^\nu_K$ is one of the seminorms defined above for $C^\nu(M)$. The construction of such topology can be found in \cite{Michor1978} and \cite{ Hirsch1997}.

In a like manner, the space of jets of mappings, denoted by $\Gamma(J_m(M;N))$, can be equipped with the initial topology defined by the family of mappings
\begin{eqnarray*}
\Psi^m_f:\Gamma(J_m(M;N))&\rightarrow&  \Gamma(J_m(M;\mathbb{R}))\\
	j_m(\Phi)&\mapsto& j_m(\Phi^*f)
\end{eqnarray*}
for $f\in C^r(N)$. Thus this topology is defined by the family of semimetrics
$$d^{m,\nu}_{K,f}(j_m\Phi_1, j_m\Phi_2)= \sup\{p^{m,\nu}_K(j_m(f\circ\Phi_1)(x)-j_m(f\circ \Phi_2)(x))\, | \, x\in K\},$$
where $\,\,\, f\in C^r(N), \, K\subseteq M$ compact, and $p^{m,\nu}_K$ is a seminorms for $\Gamma(J_m(M;\mathbb{R}))$, described above. 

\subsubsection{Topologies for space of time- and parameter-dependent vector fields}
We provide a methodology for working with vector fields with measurable time-dependence and continuous parameter-dependence for the resulting flows of such vector fields. The notion of integrability (with time) we use is ``integrability by seminorm," which seems to originate in \cite{Garnir(1972)}. For Suslin spaces, such as we are working with, integrability by seminorm amounts to the requirement that the application of any continuous seminorm to the vector-valued functions yields a function in the usual scalar $\text{L}^1$-space. 

\textbf{Time-dependent vector fields.} Let $m\in \mathbb{Z}_{\geq 0}$, let $m'\in\{0,\text{lip}\}$, let $\nu\in\{m+m',\infty,\omega,\text{hol}\}$, and let $r\in \{\infty,\omega,\text{hol}\}$, as required. Let $\pi : E \rightarrow M$ be a $C^r$-vector bundle and let $\mathbb{T} \subseteq \mathbb{R}$ be an interval. We denote by $\text{L}^1_{\text{loc}}(\mathbb{T}; \Gamma^\nu(E))$, the space of time-varying sections which are locally Bochner integrable with respect to time. The seminorms on $\text{L}^1_{\text{loc}}(\mathbb{T}; \Gamma^\nu(E))$ can be provided as
\begin{equation*}
    p_{K,\mathbb{S},1}(X)=\int_{\mathbb{S}}p^\nu_K\circ X_t \;~d t,\hspace{10pt} \mathbb{S}\subseteq\mathbb{T} \text{ and } K\subseteq M \text{ compact.}
\end{equation*}
This family of seminorms makes $\text{L}^1_{\text{loc}}(\mathbb{T}; \Gamma^\nu(E))$ a locally convex space.
We commment that $$\text{L}^{1}_{\text{loc}}(\mathbb{T};\Gamma^{\nu}(E))\simeq \text{L}^{1}(\mathbb{T};\mathbb{R})\widehat{\otimes}_\pi \Gamma^{\nu}(E),$$
where $\widehat{\otimes}_\pi$ denotes the completed projective tensor product, cf.~\cite[Theorem 3.2]{Andrew(2022)}.

\textbf{Time- and parameter-dependent vector fields.} Let $\mathcal{P}$ be a topological space and consider a map $\xi:\mathbb{T}\times M\times \mathcal{P}\rightarrow E$ with the property that $\xi(t,x,p)\in E_x$ for each $(t,x,p)\in\mathbb{T}\times M\times \mathcal{P}$. Denote by $\xi^p:\mathbb{T}\times M\rightarrow E$ the map $\xi^p(t,x)=\xi(t,x,p)$. We denote $C^0(\mathcal{P};\text{L}^{1}_{\loc}(\mathbb{T};\Gamma^{\nu}(E)))$ the space of time- and parameter-dependent sections with the property that $\mathcal{P}\ni p\to \xi^p\in\text{L}^{1}_{\loc}(\mathbb{T};\Gamma^{\nu}(E))$ is continous. 

More explicitly, for members in $C^0(\mathcal{P};\text{L}^{1}_{\loc}(\mathbb{T};\Gamma^{\nu}(E)))$, we note that the conditions for such membership on $\xi$ are, just by definition: for each $p_0 \in \mathcal{P}$, for each compact $K \subseteq M$ and $\mathbb{S} \subseteq \mathbb{T}$, and for each $\epsilon \in \mathbb{R}_{>0}$, there exists a neighbourhood $\mathcal{O} \subseteq \mathcal{P}$ of $p_0$ such that
\begin{equation}\label{eq:2.4}
    \int_{\mathbb{S}}p^\nu_K(\xi^p_t-\xi^{p_0}_t)~d t<\epsilon,\hspace{10pt}p\in\mathcal{O}.
\end{equation}

\subsection{Topologies for local flows}\label{section2.2}
We shall study flows in two settings: flows solely defined on their own (codomain of map (\ref{exp})) and flows for vector fields (image of map (\ref{exp})). We shall define the first setting in this section and establish an appropriate topology based on the properties of flows. 

First we define what we mean by a absolutely continuous curve on a manifold, as this will be used in an essential way in our characterisation of flows.
\begin{defin}[Absolutely continuous curve]
    Let $\mathbb{T} \subseteq \mathbb{R}$ be an interval and let $M$ be a $C^r$-manifold. We assume that $M$ is Stein if $\nu=\text{hol}$. A curve $\xi : \mathbb{T} \to M$ is absolutely continuous if $f \circ\xi: \mathbb{T} \to\mathbb{R}$ is absolutely continuous for every $f \in C^r(M)$.
\end{defin}
\begin{defin}[$C^\nu$-local flow]\label{def:2.1}
Let $m\in \mathbb{Z}_{\geq 0}$, let $m'\in\{0,\text{lip}\}$, let $\nu\in\{m+m',\infty,\omega,\text{hol}\}$, and let $r\in \{\infty,\omega,\text{hol}\}$, as required. Let $M$ be a $C^r$-manifold, let $\mathbb{T}\subseteq\mathbb{T'}\subseteq \mathbb{R}$ be time intervals, and let $\mathcal{U}\subseteq M$ be open. We denote by $\text{LocFlow}^\nu(\mathbb{T'};\mathbb{T};\mathcal{U})$ the space of local flows $\Phi:\mathbb{T'}\times\mathbb{T}\times \mathcal{U}\rightarrow M$ with the following properties:
\begin{enumerate}
    \item $\Phi(t_0,t_0,x)=x,\;(t_0,t_0,x)\in \mathbb{T'}\times\mathbb{T}\times \mathcal{U}$;
    \item $\Phi(t_2,t_1,\Phi(t_1,t_0,x))=\Phi(t_2,t_0,x)$, $t_0,t_1,\in\mathbb{T},\; t_2\in\mathbb{T'}\;x\in\mathcal{U}$;
    \item the map $x\mapsto\Phi(t_1,t_0,x)$ is a $C^\nu$-diffeomorphism onto its image for every $t_0\in\mathbb{T}$ and $t_1\in\mathbb{T'}$;
	\item the map $t_0\mapsto\Phi_{t_1,t_0}\in C^\nu(\mathcal{U};M)$ is continuous for a fixed $t_1\in\mathbb{T'}$, and the map $t_1\mapsto\Phi_{t_1,t_0}\in C^\nu(\mathcal{U};M)$ is absolutely continuous for a fixed $t_0\in\mathbb{T}$, where $\Phi_{t_1,t_0}(x)=\Phi(t_1,t_0,x)$.
\end{enumerate}    
\end{defin}
\textbf{Topologies for parameter-independent local flows.} Note that a local flow $\Phi\in \text{LocFlow}^\nu(\mathbb{T'};\mathbb{T};\mathcal{U})$ defines absolutely continuous curves with respect to its final time,
$$\hat{\Phi}\in \rm{AC}(\mathbb{T'};C^0(\mathbb{T};C^\nu(\mathcal{U};M))),$$
by $\hat{\Phi}(t)(t_0)(x)=\Phi(t,t_0,x)$. Motivatied by this, we topologise $\text{LocFlow}^\nu(\mathbb{T'};\mathbb{T};\mathcal{U})$ as follows.

First, we give $C^0(\mathbb{T};C^\nu(\mathcal{U}))$ the compact-open topology defined by the seminorms 
$$p^\nu_{K,\mathbb{I}}(g)=\sup\{p^\nu_K\circ g(t_0)\ |\ t_0\in\mathbb{I})\},$$
where $\mathbb{I}\subseteq\mathbb{T}$ a compact interval and $p^\nu_K$ is the appropriate seminorm for $C^\nu(\mathcal{U})$. Second, we give $\rm{AC}(\mathbb{T'};C^0(\mathbb{T};C^\nu(\mathcal{U})))$ the topology defined by the seminorms 
$$q^\nu_{K,\mathbb{I},\mathbb{I'}}(g)=\max\{p^\nu_{K,\mathbb{I},\mathbb{I'},\infty}(g),\ \hat{p}^\nu_{K,\mathbb{I},\mathbb{I'},1}(g)\},$$
where 
$$p^\nu_{K,\mathbb{I},\mathbb{I'},\infty}(g)=\sup\{p^\nu_{K,\mathbb{I}}\circ g(t)\ |\ t\in\mathbb{I'}\}\hspace{5pt}\text{ and }\hspace{5pt} \hat{p}^\nu_{K,\mathbb{I},\mathbb{I'},1}(g)=\int_{\mathbb{I'}}p^\nu_{K,\mathbb{I}}\circ g'(t) dt,$$
$\mathbb{I}\subseteq\mathbb{T}$, $\mathbb{I'}\subseteq\mathbb{T'}$ compact intervals. Note that, for $f\in C^r(M)$, the composition $f\circ\Phi\in \rm{AC}(\mathbb{T'};C^0(\mathbb{T};C^\nu(\mathcal{U})))$ and the topology of this space can be given more explicitly by the maximum of 
$$p_{K,\mathbb{I},\mathbb{I'},\infty}^{\nu}(f\circ \Phi)=\sup \left\{p_{K}^{\nu}(f\circ\Phi_{t_1,t_0})\ |\ (t_1,t_0)\in\mathbb{I'}\times\mathbb{I}\right\}$$
and 
\begin{eqnarray*}
   \hat{p}^\nu_{K,\mathbb{I},\mathbb{I'},1}(f\circ\Phi)&=&\int_{\mathbb{I'}} p^\nu_{K,\mathbb{I}} \left(\frac{d}{d t}(f\circ\Phi_{t,t_0})\right)d t\\
   &=& \int_{\mathbb{I'}} p^\nu_{K,\mathbb{I}} \left(\langle d f(\Phi_{t,t_0}),\frac{d}{d t}\Phi_{t,t_0}\rangle\right) dt\\
   &=& \int_{\mathbb{I'}} \sup\left\{p^\nu_K\left(\langle d f(\Phi_{t,t_0}),\frac{d}{d t}\Phi_{t,t_0}\rangle\right)\ \bigg|\ t_0\in\mathbb{I}\right\} d t,
\end{eqnarray*}
where $K\subseteq M$ and $\mathbb{I}\subseteq\mathbb{I'}\subseteq \mathbb{T}$ are compact.

Finally, we give $\rm{AC}(\mathbb{T'};C^0(\mathbb{T};C^\nu(\mathcal{U};M)))$ the initial topology associated with the mappings
\begin{eqnarray*}
    \Psi_f:\rm{AC}(\mathbb{T'};C^0(\mathbb{T};C^\nu(\mathcal{U};M)))&\rightarrow& \rm{AC}(\mathbb{T'};C^0(\mathbb{T};C^\nu(\mathcal{U})))\\
       \Phi&\mapsto& f\circ\Phi.
\end{eqnarray*}
for $f\in C^r(M)$.

\textbf{Parameter-dependent local flows.} Let $\mathcal{P}$ be a topological space. We denote by $\text{LocFlow}^\nu(\mathbb{T'};\mathbb{T};\mathcal{U};\mathcal{P})$ the set of parameter dependent local flows $\Phi^p$ such that 
$$\mathcal{P}\ni p\mapsto \Phi^p\in \text{LocFlow}^\nu(\mathbb{T'};\mathbb{T};\mathcal{U})$$
are continuous, i.e., $C^0(\mathcal{P};\text{LocFlow}^\nu(\mathbb{T'};\mathbb{T};\mathcal{U}))$.

Notice we have used the following analogous notations to avoid the abuse of notation:
\begin{enumerate}
    \item $\text{LocFlow}^\nu(\mathbb{T'};\mathbb{T};\mathcal{U})=\rm{AC}(\mathbb{T'};C^0(\mathbb{T};C^\nu(\mathcal{U};M)))$;
    \item $\text{LocFlow}^\nu(\mathbb{T'};\mathbb{T};\mathcal{U};\mathcal{P})= C^0(\mathcal{P};\rm{AC}(\mathbb{T'};C^0(\mathbb{T};C^\nu(\mathcal{U};M))))$.
\end{enumerate}
\section{Flows for vector fields}\label{section3}
We have defined, very generally, the space of local flows and  its topologies in Section \ref{section2.2}. However, the question arises that for $X\in C^0(\mathcal{P};\text{L}^{1}_{\loc}(\mathbb{T};\Gamma^{\nu}(TM)))$, if its flow $\Phi^X\in\text{LocFlow}^\nu(\mathbb{S'};\mathbb{S};\mathcal{U};\mathcal{P})$ for some $\mathbb{S}\subseteq\mathbb{S'}\subseteq\mathbb{T}$. To answer this question, we shall give geometric definitions and characterisations of integral curves and flows for vector fields. 

\subsection{Integral curves and flows for vector fields}
A curve $\xi:\mathbb{S'}\rightarrow M$ for $X^p\in C^0(\mathcal{P};\text{L}^{1}_{\loc}(\mathbb{T};\Gamma^{\nu}(TM)))$ is an \textbf{integral curve} if
\begin{enumerate}
    \item $\xi$ is locally absolutely continuous and $\xi'(t)=X^p(t,\xi(t))$ for almost every $t\in\mathbb{S'}$, $\mathbb{T'}\subseteq \mathbb{T}$,
    \item or equivalently,
    $$f\circ \xi(t)=f\circ\xi(t_0)+\int^{t}_{t_0}X^pf(s,\xi(s))ds$$
    for each $t_0\in\mathbb{T'}$ and each $f\in C^r(M)$. 
\end{enumerate}
We denote 
\begin{multline*}
    D_X=\{\mathbb{S'}\times\mathbb{S}\times\mathcal{U}\times \mathcal{O}\subseteq\mathbb{T}\times\mathbb{T}\times M\times\mathcal{P}\;|\; \text{ there exists an integral curve} \\
    \xi:\mathbb{S'}\to M \text{ for } X^{p_0} \text{ such that } \xi(t_0)=x_0 \text{ for each } (t_0,x_0,p_0)\in \mathbb{S}\times\mathcal{U}\times\mathcal{O} \}.
\end{multline*}
The \textbf{\textit{flow}} for $X$ is the mapping
\begin{eqnarray*}
\Phi^X:D_X &\rightarrow&  M\\
	         (t_1,t_0,x_0,p_0)&\mapsto& \xi(t_1),
\end{eqnarray*}
where $\xi$ is the integral curve for $X^{p_0}$ satisfying $\xi(t_0)=x_0$. Hence the flow $\Phi^X$ satisfies 
$$f\circ \Phi^X(t,t_0,x,p) = f(x)+\int^{t}_{t_0}X^pf(s,t_0,x)d s$$
for each $t_0\in\mathbb{T'}$ and each $f\in C^r(M)$.

We will work with time-varying vector fields, both with and without parameter dependence. When we work with vector fields that are time-dependent but not parameter-dependent, we will simply omit the argument corresponding to the parameter without further mention. We shall also use the notation
$$\Phi^{X^p}_{t_1,t_0}:D_X(t_1,t_0,p)\rightarrow M$$
when convenient.

We make an important remark that the parameter-independent flows $\Phi^{X^p}$ (for fixed $p$) for vector fields $X^p$ have elementary properties that coincide with Definition \ref{def:2.1}. This is to say, the flows for vector fields (image of map (\ref{exp})) are contained in the space of flows defined solely on its own (Section \ref{section2.2}, codomain of map (\ref{exp})). This is also true for the parameter-dependent case, which will be shown in Section \ref{section4}. 
\subsection{Compactness of image of flows}
In this section and the next, we establish some technical results that will be useful in Section \ref{section4}. This first result gives an important compactness result. We emphasize there the fact that there are no restrictions on the topological space $\mathcal{P}$.
\begin{lemma}[Robustness of compactness by variations of parameters]\label{cor:3.11}
Let $M$ be a $C^\infty$-manifold, let $\mathbb{T} \subseteq \mathbb{R}$ be an interval, let $\mathcal{P}$ be a topological space, and let $X \in C^0(\mathcal{P};\rm{L}^{1}_{\loc}(\mathbb{T};\Gamma^{\lip}(TM)))$. Let $K \subseteq M$ be compact, let $t_0, t_1 \in \mathbb{T}$, and let $p_0 \in \mathcal{P}$ be such that
\begin{equation*}
    [t_0, t_1] \times \{t_0\} \times K\times\{p_0\} \subseteq D_X. 
\end{equation*}
Then there exists a neighbourhood $\mathcal{O} \subseteq \mathcal{P}$ of $p_0$ such that
\begin{equation*}
    \bigcup_{(t,x,p)\in[t_0,t_1]\times K\times\mathcal{O}} \Phi^X(t,t_0,x,p)
\end{equation*}
is well-defined and precompact.
\end{lemma}
\begin{proof}
Denote
$$K_0:=\bigcup_{(t,x)\in[t_0,t_1]\times K} \Phi^{X}(t,t_0,x,p_0).$$
Since $[t_0,t_1]\times t_0\times K\times p_0$ is compact and $\Phi^{X}:D_X\to M$ is continuous, we have that $K_0$ is compact. Since $M$ is locally compact, let $\mathcal{V}$ be a precompact neighbourhood of $K_0$. For $x \in K$, let $\mathcal{U}_x$ be a neighbourhood of $x$ and let $\mathcal{O}_x$ be a neighbourhood of $p_0$ such that
$$\bigcup_{(t,x,p)\in[t_0,t_1]\times (K\cap\mathcal{U}_x)\times\mathcal{O}_x} \Phi^X(t,t_0,x,p)\subseteq\mathcal{V}.$$
By compactness of $K$, let $x_1, . . . , x_m \in K$ be such that
$K=\cup_{j=1}^{m}K\cap \mathcal{U}_{x_j}$ and let $\mathcal{O}=\cap_{j=1}^{k}\mathcal{O}_{x_j}$. Then 
$$\Phi^X(t,t_0,x,p)\subseteq\mathcal{V},\hspace{10pt}(t,x,p)\in[t_0,t_1|]\times K\times\mathcal{O},$$
as desired.
\end{proof} 
\subsection{Integrable sections along a curve}\label{appendixb}
We denote by $\text{Aff}^r(E) \subseteq C^r(E)$ the set of $C^r$-functions $G$ on the manifold $E$ for which $G|E_x$ is affine for each $x \in M$. We denote by $L^1_{\loc}(\mathbb{T}; E)$ the mappings $\Gamma: \mathbb{T} \rightarrow E$ for which $G \circ \Gamma \in L^1_{\loc}(\mathbb{T};\mathbb{F})$ for every $G \in \text{Aff}^r(E)$. Note that, if $\Gamma \in L^1_{\loc}(\mathbb{T}; E)$, then there is a mapping $\gamma : \mathbb{T} \rightarrow M$ specified by the requirement that the diagram
\[\begin{tikzcd}
\mathbb{T} \arrow{dr}{\gamma} \arrow{r}{\Gamma}
& E \arrow{d}{\beta} \\
& M
\end{tikzcd}\]
commutes. We can think of $\Gamma$ as being a section of $E$ over $\gamma$. We topologise $L^1_{\loc}(\mathbb{T}; E)$ by
giving it the initial topology associated with the mappings
\begin{eqnarray*}
  \alpha_F:L^1_{\loc}(\mathbb{T}; E)&\rightarrow& L^1(\mathbb{S};\mathbb{F})\\
  \Gamma&\mapsto& G\circ\Gamma|\mathbb{S},
\end{eqnarray*}
where $G\in\text{Aff}^r(E)$ and $\mathbb{S}\subseteq\mathbb{T}$ a compact subinterval.

Associated to these spaces of integrable sections over a curve, we have a few constructions and technical results whose importance will be made apparent at various points during the subsequent presentation. We consider the space $C^0 (\mathbb{T}; M)$ with the topology (indeed,
uniformity) defined by the family of semimetrics 
\begin{equation}\label{eq:gam12}
    d_{\mathbb{S},M}(\gamma_1, \gamma_2) = \sup\{d_\mathbb{G}(\gamma_1(t), \gamma_2(t))\ |\  t \in \mathbb{S}\}, \hspace{10pt} \mathbb{S} \subseteq \mathbb{T} \text{ a compact interval.}
\end{equation}
We consider this in the following context. We consider $C^r$-vector bundles $\pi_E : E \rightarrow M$ and $\pi_F : F \rightarrow N$. We abbreviate
$$E^* \otimes F = \text{pr}^*_1 E^* \otimes \text{pr}^*_2 F,$$
where $\text{pr}_1: M \times N \rightarrow M$ and $\text{pr}_2: M \times N \rightarrow N$ are the projections. We regard $E^* \otimes F$ as a
vector bundle over $M \times N$. The fibre over $(x, y)$ we regard as
$$E^*_x \otimes F_y \simeq \text{Hom}_\mathbb{F}(E_x; F_y).$$
Note that the total spaces $E$ and $F$ of these vector bundles, and so also $E^* \otimes F$, inherit a Riemannian metric from a Riemannian metrics on their base spaces and fibre metrics. Thus $E^*\otimes F$ possess the associated distance function, and we shall make use of this to define, as in (\ref{eq:gam12}), a topology on the space $C^0(\mathbb{T}; E^*\otimes F)$. If $\Gamma \in C^0(\mathbb{T}; E^*\otimes F)$, then we have induced mappings
$$\gamma_M \in C^0(\mathbb{T}; M),\hspace{10pt} \gamma_N \in C^0(\mathbb{T}; N)$$
obtained by first projecting to $M \times N$, and then projecting onto the components of the product. Note that $\Gamma(t) \in \text{Hom}_\mathbb{F}(E_{\pi_E\circ\Gamma(t)}; F_{\pi_F\circ\Gamma(t)}),\  t \in \mathbb{T}$.

If $\xi : \mathbb{T} \times M \rightarrow E$ satisfies $\xi(t, x) \in E_x$ and if $\Gamma: \mathbb{T} \rightarrow E^* \otimes F$, then we can define
\begin{eqnarray*}
  \xi_\Gamma:\mathbb{T}&\to&F\\
  t&\mapsto&\Gamma(t)(\xi(t,\gamma_M(t))).
\end{eqnarray*}
We call $\xi_\Gamma$ the composite section associated with $\xi$ and $\Gamma$. The following lemma shows that this mapping is integrable, under suitable hypotheses on $\xi$ and $\Gamma$.
\begin{lemma}[Integrability of composite section]\label{lem:2.5}
Let $\pi_E : E \to M$ and $\pi_F : F \to M$ be $C^\infty$-vector bundles and let $\mathbb{T} \subseteq \mathbb{R}$ be an interval. If $\xi \in \Gamma^1_{\text{loc}}(\mathbb{T}; \Gamma^\nu(E))$ and if $\Gamma \in C^0(\mathbb{T}; E^*\otimes F)$, then $\xi_\Gamma \in L^1_{\loc}(\mathbb{T}; F)$.
\end{lemma}
\begin{proof}
We first show that $t\mapsto G\circ\xi_\Gamma$ is measurable on $\mathbb{T}$. Note that 
$$t\mapsto G\circ\Gamma(s)(\xi(t,\gamma_M(s)))$$
is measurable for each $s \in \mathbb{T}$ and that
\begin{equation}\label{eq:2.5}
    s\mapsto G\circ\Gamma(s)(\xi(t,\gamma_M(s)))
\end{equation}
is continuous for each $t \in \mathbb{T}$. Let $[a, b] \subseteq \mathbb{T}$ be compact, let $k \in \mathbb{Z}_{>0}$, and denote
$$t_{k,j} = a + \frac{j-1}{k}(b - a), \hspace{10pt} j \in \{1, . . . , k + 1\}.$$
Also denote 
$$\mathbb{T}_{k,j} = [t_{k,j} , t_{k,j+1}), \hspace{10pt} j \in \{1, . . . , k - 1\},$$
and $\mathbb{T}_{k,k} = [t_{k,k}, t_{k,k+1}]$. Then define $g_k : \mathbb{T} \rightarrow \mathbb{R}$ by
$$g_k(t)=\sum_{j=1}^{k}G\circ\Gamma(t_{k,j})(\xi(t,\gamma_M(t_{k,j})))\chi_{t_{k,j}}.$$
Note that $g_k$ is measurable, being a sum of products of measurable functions (\cite[Proposition 2.1.7]{Cohn(2013)}). By continuity of (\ref{eq:2.5}) for each $t \in \mathbb{T}$, we have
$$\lim_{k\to\infty}g_k(t) = G\circ\Gamma(t)(\xi(t,\gamma_M(t))),\hspace{10pt} t \in [a, b],$$
showing that $t\mapsto G\circ\Gamma(t)(\xi(t,\gamma_M(t)))$ is measurable on $[a,b]$, as pointwise limits of measurable functions are measurable \cite[Proposition 2.1.5]{Cohn(2013)}. Since the compact interval $[a, b] \subseteq \mathbb{T}$ is arbitrary, we conclude that $t \mapsto  G\circ\Gamma(t)(\xi(t,\gamma_M(t)))$ is measurable on $\mathbb{T}$.

Let $\mathbb{S} \subseteq \mathbb{T}$ be compact and let $K \subseteq M$ be a compact set for which $\gamma_M(\mathbb{S}) \subseteq K$. Since $\xi\in \Gamma^0_{\text{LI}}(\mathbb{T};M)$ and since $\Gamma$ is continuous, there exists $h\in L^1(\mathbb{S};\mathbb{R}_{\geq 0})$ be such that 
$$|G\circ \Gamma(t)(\xi(t,x))|\leq h(t) \hspace{20pt} (t,x)\in\mathbb{S}\times K,$$
In particular, this shows that $t \mapsto G\circ\Gamma(t)(\xi(t,\gamma_M(t)))$ is integrable on $\mathbb{S}$ and so locally integrable on $\mathbb{T}$.
\end{proof}
The following simplified version of the lemma will be useful.
\begin{cor}[Integrability of composite section]\label{cor:2.7}
Let $M$ be a $C^\infty$-manifold and let $\mathbb{T} \subseteq \mathbb{R}$ be an interval. If $f \in L^1_{\text{loc}}(\mathbb{T}; C^\nu(M))$, if $\gamma \in C^0(\mathbb{T}; M)$, and if we define $f_\gamma : \mathbb{T} \to \mathbb{R}$ by $f_\gamma(t) = f(t, \gamma(t))$, then $f_\gamma \in L^1_{\rm{loc}}(\mathbb{T}; \mathbb{R})$.
\end{cor}
\begin{proof}
Apply the lemma with $E = F = M\times \mathbb{R}$ (so that sections are identified with functions) and $\Gamma(t) = ((\gamma(t), \gamma(t)),\text{id}_\mathbb{R})$.
\end{proof}
We also have the mapping
\begin{eqnarray*}
  \Psi_{\mathbb{T},E,\xi}:C^0(\mathbb{T}; E^* \otimes F) &\to& L^1_{\loc}(\mathbb{T}; F)\\
  \Gamma&\mapsto& \xi_\Gamma,
\end{eqnarray*}
which is well-defined by Lemma \ref{lem:2.5}. The following lemma gives the continuity of this mapping.
\begin{lemma}[Continuity of curve to composite section map]\label{lem:2.6}
Let $\pi_E : E \to M$ and $\pi_F : F\to M$ be $C^\infty$-vector bundles, let $\mathbb{T} \subseteq \mathbb{R}$ be an interval, and let $\xi \in \Gamma^0_{\text{LI}}(\mathbb{T}; E)$. Then $\Psi_{\mathbb{T},E,\xi}$ is continuous.
\end{lemma}
\begin{proof}
Let $G \in \text{Aff}^\infty(F)$ and let $\mathbb{S} \subseteq \mathbb{T}$ be a compact interval. Let $\Gamma_j \in C^0(\mathbb{T}; E^* \otimes F)$,
$j \in \mathbb{Z}_{>0}$, be a sequence of curves converging to $\Gamma \in C^0(\mathbb{T}; E^* \otimes F)$. Since $\Gamma(\mathbb{S})$ is compact
and $E^* \otimes F$ is locally compact, we can find a precompact neighbourhood $\mathcal{W}$ of $\Gamma(\mathbb{S})$. Then, for $N \in \mathbb{Z}_{>0}$ sufficiently large, we have $\Gamma_j (\mathbb{S}) \subseteq \mathcal{W}$, $j\geq N$ by uniform convergence. Therefore, we can find a compact set $L \subseteq E^* \otimes F$ such that $\Gamma^j (\mathbb{S}) \subseteq L$, $j \in \mathbb{Z}_{>0}$, and $\Gamma(\mathbb{S}) \subseteq L$. Let $g \in L^1(\mathbb{S}; \mathbb{R}_{\geq0})$ be such that
$$|G\circ\Gamma(t)(\xi(t,x))|\leq h(t) \hspace{10pt} (t,x)\in\mathbb{S}\times K,$$
this since $\xi \in \Gamma^0_{\text{LI}}(\mathbb{T}; M)$ and since $\Gamma$ is continuous. Then, for fixed $t \in \mathbb{S}$, continuity of
$x \mapsto G \circ \Gamma(t)(\xi(t, x))$ ensures that
$$\lim_{j\to\infty}G\circ\Gamma(t)(\xi(t,\gamma_{M,j}(t)))= G\circ\Gamma(t)(\xi(t,\gamma_M(t))).$$
We also have
$$|G\circ\Gamma(t)(\xi(t,\gamma_{M,j}(t)))|\leq g(t),\hspace{10pt} t\in\mathbb{S}.$$
Therefore, by the Dominated Convergence Theorem \cite[Theorem 2.4.5]{Cohn(2013)}
$$\lim_{j\to\infty}\int_{\mathbb{S}}G\circ\Gamma(t)(\xi(t,\gamma_{M,j}(t)))~d t=\int_{\mathbb{S}}G\circ\Gamma(t)(\xi(t,\gamma_M(t)))~dt,$$
which gives the desired continuity.
\end{proof}
\section{Continuous dependence of fixed-time flow on parameter}\label{section4}
To answer the question we proposed at the beginning of Section \ref{section3}, that, for $X\in C^0(\mathcal{P};\text{L}^{1}_{\loc}(\mathbb{T};\Gamma^{\nu}(TM)))$, if its flow $\Phi^X\in\text{LocFlow}^\nu(\mathbb{S'};\mathbb{S};\mathcal{U};\mathcal{P})$, we will give the following theorem. 
\begin{thm}\label{thm;3.4}
Let $m\in \mathbb{Z}_{\geq 0}$, let $\nu\in\{m,\infty,\rm{hol}\}$ satisfy $\nu\geq \rm{lip}$, and let $r\in \{\infty,\rm{hol}\}$as appropriate. Let $M$ be a $C^r$-manifold, let $\mathbb{T}\subseteq\mathbb{R}$ be an interval, and let $\mathcal{P}$ be a topological space. For $X\in C^0(\mathcal{P};\rm{L}^{1}_{\loc}(\mathbb{T};\Gamma^{\nu}(TM)))$, let $t_0,t_1\in\mathbb{T}$ and $p_0\in\mathcal{P}$ be such that there exists a precompact open set $\mathcal{U}\subseteq M$ such that $\rm{cl}(\mathcal{U})\subseteq \rm{D}_X(t_1,t_0,p_0)$. Then there exists a neighbourhood $\mathcal{O}\subseteq\mathcal{P}$ of $p_0$ such that mapping
\begin{eqnarray*}
  \mathcal{O}\ni p \mapsto \Phi^{X^p}_{t_1,t_0}\in C^\nu(\mathcal{U};M)
\end{eqnarray*}
is well-defined and continuous.
\end{thm}
\begin{proof} 
We break down the proof into the various classes of regularity. The proofs bear a strong resemblance to one another, so we go through the details carefully in the first case we prove, the locally Lipschitz case, and then merely outline where the arguments differ for the other regularity classes.

\textbf{The $\mathbf{C^0}$-case.} Note that $\rm{cl}(\mathcal{U})$ is compact. Therefore, by Lemma \ref{cor:3.11}, there exists a compact set $K' \subseteq M$ and a neighbourhood $\mathcal{O}$ of $p_0$ such that
$$\Phi^{X^p}_{t_1,t_0}(x)\in \text{int}(K'),\hspace{10pt}(x,p)\in\text{cl}(\mathcal{U})\times \mathcal{O}.$$
This gives the well-definedness assertion of the theorem. Note, also, that it gives the well-definedness assertion for all $\nu \in \{m, \infty, \omega, \text{hol}\}$, and so we need not revisit this for the remainder of the proof. The compact set $K' \subseteq M$ and the neighbourhood $\mathcal{O}$ of $p_0$ will be used in all parts of the proof without necessarily referring to our constructions here.

For continuity, first we show that the mapping
\begin{eqnarray*}
  \mathcal{O}\ni p \mapsto \Phi^{X^p}_{t_1,t_0}\in C^0(\mathcal{U};M)
\end{eqnarray*}
is continuous. The topology for $C^0(\mathcal{U};M)$ is the uniform topology defined by the semimetrics 
$$d^0_{K,f}(\Phi_1,\Phi_2)=\sup\{|f\circ\Phi_1(x)-f\circ\Phi_2(x)|\ | \ x\in K\}$$
where  $f\in C^\infty(M)$ and $K\subseteq \mathcal{U}$ compact.
Thus, we must show that, for $f_1,...,f_m\in C^\infty(M)$, for $K_1,...,K_m\subseteq\mathcal{U}$ compact, and for $\epsilon_1,...,\epsilon_m\in\mathbb{R}_{>0}$, there exists a neighborhood $\mathcal{O}$ of $p_0$ such that 
$$|f_j\circ\Phi^{X^p}_{t_1,t_0}(x_j)-f_j\circ \Phi^{X^{p_0}}_{t_1,t_0}(x_j)|<\epsilon_j,\hspace{10pt} x_j\in K_j, \ p\in\mathcal{O}, \ j\in\{1,...,m\}.$$
It will suffice to show that, for $f\in C^\infty(M)$, for $K\subseteq\mathcal{U}$ compact, and for $\epsilon\in\mathbb{R}_{>0}$, we have  
$$|f\circ\Phi^{X^p}_{t_1,t_0}(x)-f\circ \Phi^{X^{p_0}}_{t_1,t_0}(x)|<\epsilon,\hspace{10pt} x\in K, \ p\in\mathcal{O}.$$
Indeed, if we show that then, taking $K=\cup_{j=1}^{k}K_j$ and $\epsilon
=\min\{\epsilon_1,...,\epsilon_m\}$, we have 
$$|f_j\circ\Phi^{X^p}_{t_1,t_0}(x)-f_j\circ \Phi^{X^{p_0}}_{t_1,t_0}(x)|<\epsilon,\hspace{10pt} x\in K, \ p\in\mathcal{O}, \ j\in\{1,...,m\}$$
for a suitable $\mathcal{O}$. This suffices to give the desired conclusion.

It is useful to consider the space $C^0(\mathbb{T};M)$ with the topology (indeed, uniformity) defined by the family of semimetrics
$$d_{\mathbb{S},M}(\gamma_1,\gamma_2)=\{d_{\mathbb{G}}(\gamma_1(t),\gamma_2(t))\ |\ t\in \mathbb{S}\},\hspace{10pt}\mathbb{S}\subseteq\mathbb{T} \text{ a compact interval.}$$
For $g\in C^0_{\text{LI}}([t_0,t_1];M)$, we also have the mapping
\begin{eqnarray*}
  \Psi_{[t_0,t_1],M,g}:C^0([t_0,t_1];M) &\rightarrow & L^1_{\rm{loc}}([t_0,t_1];\mathbb{R})\\
       \gamma&\mapsto& (t\mapsto g_t(\gamma(t))).
\end{eqnarray*}
By Lemma \ref{cor:2.7} and Lemma \ref{lem:2.6},  this map is well-defined and continuous.
\renewcommand\qedsymbol{$\nabla$}
We also have the continuous mapping
\begin{eqnarray*}
  \Phi_{[t_0,t_1], K,\mathcal{O}}:K\times\mathcal{O} &\rightarrow & C^0([t_0,t_1];M)\\
       (x,p)&\mapsto& (t\mapsto\Phi^X(t,t_0,x,p)).
\end{eqnarray*}
Let $f\in C^\infty(M)$, let $\epsilon>0$, and let $x\in K$. Combining the observations of two previous paragraphs, the mapping
$$\Psi_{[t_0,t_1],M,X^{p_0}f}\circ \Phi_{\mathbb{T}, K,\mathcal{O}}: K\times \mathcal{O}\rightarrow L^1([t_0,t_1];\mathbb{R})$$
is continuous. Thus there exists a relative neighbourhood $\mathcal{V}_x\subseteq K$ of $x$ and a neighbourhood $\mathcal{O}_x\subseteq\mathcal{O}$ of $p_0$ such that
$$\int_{t_0}^{t_1}|X^{p_0}f(s,\Phi^{X^p}_{s,t_0}(x'))-X^{p_0}f(s,\Phi^{X^{p_0}}_{s,t_0}(x'))|~ds<\frac{\epsilon}{2} \hspace{15pt} x'\in\mathcal{V}_x, \ p\in\mathcal{O}_x.$$
Let $x_1,...,x_m\in K$ be such that $K= \cup_{j=1}^{m}\mathcal{V}_{x_j}$ and define a neighbourhood $\mathcal{O}_1=\cap_{j=1}^{k}\mathcal{O}_{x_j}$ of $p_0$. Then we have
\begin{equation}\label{eq:5.19}
    \int_{t_0}^{t_1}|X^{p_0}f(s,\Phi^{X^p}_{s,t_0}(x))-X^{p_0}f(s,\Phi^{X^{p_0}}_{s,t_0}(x))|\ ds<\frac{\epsilon}{2},\hspace{15pt}x\in K,\ p\in\mathcal{O}_1 .
\end{equation}
By (\ref{eq:2.4}), we can further shrink $\mathcal{O}_1$ if necessary so that
$$\int_{t_0}^{t_1}|X^{p}f(s,x)-X^{p_0}f(s,x)|\ ds<\frac{\epsilon}{2} \hspace{15pt} x'\in K, \ p\in\mathcal{O}_1.$$
Then we have 
\begin{eqnarray*}
   &&|f\circ\Phi^{X^p}_{t_1,t_0}(x)-f\circ \Phi^{X^{p_0}}_{t_1,t_0}(x)|\\
   &&\hspace{2cm}\leq\int_{t_0}^{t_1}|X^{p}f(s,\Phi^{X^p}_{s,t_0}(x))-X^{p_0}f(s,\Phi^{X^{p_0}}_{s,t_0}(x))|\ ds\\
   &&\hspace{2cm}\leq \int_{t_0}^{t_1}|X^{p}f(s,\Phi^{X^p}_{s,t_0}(x))-X^{p_0}f(s,\Phi^{X^{p}}_{s,t_0}(x))|\ ds\\
   &&\hspace{2.4cm}+\int_{t_0}^{t_1}|X^{p_0}f(s,\Phi^{X^p}_{s,t_0}(x))-X^{p_0}f(s,\Phi^{X^{p_0}}_{s,t_0}(x))|\ ds\\
   &&\hspace{2cm}\leq \frac{\epsilon}{2}+\frac{\epsilon}{2}=\epsilon,
\end{eqnarray*}
for $x\in K$ and $p\in\mathcal{O}_1$, as desired. 

Therefore, for every compact $K\subseteq\mathcal{U}$, every $f\in C^\infty(M)$, and every $\epsilon\in\mathbb{R}_{>0}$, if $p\in\mathcal{O}_1\cap\mathcal{O}_2$, then we ave 
$$p^0_K(f\circ\Phi^{X^p}_{t_1,t_0}-f\circ \Phi^{X^{p_0}}_{t_1,t_0})<\epsilon,$$
which gives the desired result.

\textbf{The $\mathbf{C^m}$-case.} The topology for $C^m(\mathcal{U};M)$ is the uniform topology defined by the semimetrics
\begin{eqnarray*}
   d^m_{K,f}(\Phi_1,\Phi_2)=\sup\{\|j_m(f\circ\Phi_1)(x)-j_m(f\circ\Phi_2)(x)\|_{\mathbb{G}_{M,\pi,m}}\ | \ x\in K\},\\
   f\in C^\infty(M), \ K\subseteq \mathcal{U}\ \text{compact}.
\end{eqnarray*}
As in the preceding section when we proved $C^0$ continuity, it suffices to show that, for
$f\in C^\infty(M)$, $K\subseteq\mathcal{U}$ compact, and for $\epsilon\in\mathbb{R}_{>0}$, there exists a neighbourhood $\mathcal{O'}$ of $p_0$ such that 
$$\|j_m(f\circ\Phi^{X^p}_{t_1,t_0})(x)-j_m(f\circ\Phi^{X^{p_0}}_{t_1,t_0})(x)\|_{\mathbb{G}_{M,\pi,m}}<\epsilon,\hspace{10pt}x\in K,\ p\in\mathcal{O'}.$$

Thus let $f\in C^\infty(M)$, $K\subseteq\mathcal{U}$ compact, and let $\epsilon\in\mathbb{R}_{>0}$. Consider the mapping
\begin{eqnarray*}
  \Phi_{[t_0,t_1], K,\mathcal{O}}:K\times\mathcal{O} &\rightarrow & C^0([t_0,t_1];J^m(\mathcal{U};M))\\
       (x,p)&\mapsto& (t\mapsto j_m\Phi^{X^p}_{t,t_0}(x)),
\end{eqnarray*}
which is well-defined and continuous. For $(x,p)\in K\times\mathcal{O}$ and for $t\in[t_0,t_1]$, we can think of $j_m\Phi^{X^p}_{t,t_0}(x))$ as a linear mapping 
\begin{eqnarray*}
  j_m\Phi^{X^p}_{t,t_0}(x):J^m(M;\mathbb{R})_{\Phi^{X^p}_{t,t_0}(x)} &\rightarrow & J^m(M;\mathbb{R})_x \\
       j_m g(\Phi^{X^p}_{t,t_0}(x))&\mapsto& j_m (g\circ\Phi^{X^p}_{t,t_0})(x).
\end{eqnarray*}
Now, fixing $(x,p)\in\mathcal{U}\times \mathcal{O}$ for the moment, recall the constructions in Lemma \ref{lem:2.5}, we consider the notation from those constructions with
\begin{enumerate}
    \item $N=M$,
    
    \item $E=F=J^m(M;\mathbb{R})$,
    
    \item $\Gamma(t)=j_m\Phi^{X^p}_{t,t_0}(x)\in \text{Hom}_{\mathbb{R}}(J^m(M,\mathbb{R})_{\Phi^{X^p}_{t,t_0}(x)};J^m(M;\mathbb{R})_x)$, and
    
    \item $\xi=j_m(X^{p_0}f)$.
\end{enumerate}
Thus, we have
$$\gamma_M(t)=\Phi^{X^p}_{t,t_0}(x),\hspace{10pt}\gamma_N(t)=x.$$
We then have the integrable section of $E=J^m(M;\mathbb{R})$ given by 
\begin{eqnarray*}
  \xi_\Gamma:[t_0,t_1] &\rightarrow & E\\
       t&\mapsto& (t\mapsto j_m (X^p_tf\circ\Phi^{X^p}_{t,t_0})(x))
\end{eqnarray*}
to obtain continuity of the mapping
\begin{eqnarray*}
  \Psi_{[t_0,t_1],J^m(M;\mathbb{R}),j_m (X^{p_0}f)}:&C^0([t_0,t_1];J^m(\mathcal{U};M)) \rightarrow  L^1_{\text{loc}}([t_0,t_1];J^m(M;\mathbb{R}))\\
       &\Gamma\mapsto (t\mapsto \Gamma(t)(j_m(X^{p_0}_tf)(\gamma_M(t)))),
\end{eqnarray*}
and so of the composition
$$\Psi_{[t_0,t_1],J^m(M;\mathbb{R}),j_m (X^{p_0}f)}\circ \Phi_{[t_0,t_1], K,\mathcal{O}}: K\times\mathcal{O}\rightarrow L^1_{\text{loc}}([t_0,t_1];J^m(M;\mathbb{R})).$$
Note that this is precisely the continuity of the mapping
$$K\times\mathcal{O}\ni (x,p)\mapsto (t\mapsto j_m(X^{p_0}f\circ \Phi^{X^p}_{t,t_0}(x)))\in L^1_{\text{loc}}([t_0,t_1];J^m(M;\mathbb{R})).$$
In order to convert this continuity into a continuity statement involving the fibre norm for $J^m(M;\mathbb{R})$, we note that for $x\in K$,  there exists a neighbourhood $\mathcal{V}_x$ and affine functions $F^1_x,..., F^{n+k}_x\in\text{Aff}^\infty(J^m(M;\mathbb{R}))$ which are coordinates for $\rho^{-1}_m(\mathcal{V}_x)$. We can choose a Riemannian metric for $J^m(M;\mathbb{R})$, whose restriction to fibres agrees with the fibre metric (\ref{eq:fibremetric}). Hence there exists $C_x\in\mathbb{R}_{>0}$ such that
$$\|j_m g_1(x')-j_m g_2(x')\|_{\mathbb{G}_{M,\pi,m}}\leq C_x|F^l_x\circ j_m g_1(x')-F^l_x\circ j_m g_2(x')|,$$
for $g_1,g_2\in C^\infty(M)$, $x'\in\mathcal{V}_x$, $l\in\{1,...,n+k\}$. By the continuity proved in the preceding paragraph, we can take a relative neighbourhood $\mathcal{V}_x\subseteq K$ of $x$ sufficiently small and a neighbourhood $\mathcal{O}_x\subseteq\mathcal{O}$ of $p_0$ such that 
$$\int_{t_0}^{t_1}|F^l_x\circ j_m(X^{p_0}f\circ \Phi^{X^p}_{t,t_0}(x'))-F^l_x\circ j_m(X^{p_0}f\circ \Phi^{X^{p_0}}_{t,t_0}(x'))|dt<\frac{\epsilon}{2C_x},$$
for all $x'\in\mathcal{V}_x$, $p\in\mathcal{O}_x$, and $l\in\{1,...,n+k\}$, by the definition of the topology for $L^1(|t_0, t_1|; J^m(M; \mathbb{R})$ in \ref{appendixb}. Therefore,
$$\int_{t_0}^{t_1}\|j_m(X^{p_0}f\circ \Phi^{X^p}_{t,t_0}(x'))-j_m(X^{p_0}f\circ \Phi^{X^{p_0}}_{t,t_0}(x'))\|_{\mathbb{G}_{M,\pi,m}} ds <\frac{\epsilon}{2}$$
for all $x'\in\mathcal{V}_x$, $p\in\mathcal{O}_x$. Now let $x_1,...,x_s\in K$ be such that $K= \cup_{r=1}^{s}\mathcal{V}_{x_r}$ and define a neighbourhood $\mathcal{O'}=\cap_{r=1}^{s}\mathcal{O}_{x_r}$ of $p_0$. Then we have 
\begin{equation}\label{eq:5.21}
    \int_{t_0}^{t_1}\|j_m(X^{p_0}_sf\circ \Phi^{X^p}_{s,t_0})(x')-j_m(X^{p_0}_sf\circ \Phi^{X^{p_0}}_{s,t_0})(x')\|_{\mathbb{G}_{M,\pi,m}} ds <\frac{\epsilon}{2}
\end{equation}
for all $x'\in K$, $p\in\mathcal{O'}$. By (\ref{eq:2.4}), we can further shrink $\mathcal{O'}$ if necessary so that 
$$\int_{t_0}^{t_1}\|j_m(X^{p}f)(s,y)-j_m(X^{p_0}f)(s,y)\|_{\mathbb{G}_{M,\pi,m}}d s<\frac{\epsilon}{2} \hspace{15pt} y'\in K', \ p\in\mathcal{O'}.$$
Then we have 
\begin{eqnarray*}
   &&\|j_m (f\circ\Phi^{X^p}_{t_1,t_0})(x)-j_m(f\circ \Phi^{X^{p_0}}_{t_1,t_0})(x)\|_{\mathbb{G}_{M,\pi,m}}\\
   &&\hspace{40pt}\leq\int_{t_0}^{t_1}\|j_m(X^{p}f\circ \Phi^{X^p}_{t,t_0})(x)-j_m(X^{p_0}f\circ \Phi^{X^{p_0}}_{t,t_0})(x)\|_{\mathbb{G}_{M,\pi,m}} ds\\
   && \hspace{40pt}\leq\int_{t_0}^{t_1}\|j_m(X^{p}f\circ \Phi^{X^p}_{t,t_0})(x)-j_m(X^{p_0}f\circ \Phi^{X^{p}}_{t,t_0})(x)\|_{\mathbb{G}_{M,\pi,m}} ds\\
   &&\hspace{50pt}+\int_{t_0}^{t_1}\|j_m(X^{p_0}f\circ \Phi^{X^p}_{t,t_0})(x)-j_m(X^{p_0}f\circ \Phi^{X^{p_0}}_{t,t_0})(x)\|_{\mathbb{G}_{M,\pi,m}} ds\\
   &&\hspace{40pt}\leq \frac{\epsilon}{2}+\frac{\epsilon}{2}=\epsilon,
\end{eqnarray*}
for $x\in K$ and $p\in\mathcal{O'}$, as desired. 

\textbf{The $\mathbf{C^\infty}$-case.} 
From the result in the $C^m$-case for $m\in\mathbb{Z}_{\geq 0}$, the mapping
$$\mathcal{O}\ni p\mapsto \Phi^{X^p}_{t_1,t_0}\in C^m(\mathcal{U};M)$$
is continuous for each $m\in\mathbb{Z}_{\geq 0}$. From the diagram 
\begin{center}
    \begin{tikzcd}[column sep=large]
    C^{\infty}(\mathcal{U};M)
    \arrow[r,"\hookrightarrow"] &C^{m}(\mathcal{U};M) 
     \\
    \mathcal{O} \arrow[u, "p\mapsto\Phi^{X^p}_{t_1,t_0}"]
    \arrow[ur, "p\mapsto j_m\Phi^{X^p}_{t_1,t_0}"', sloped]  
    \end{tikzcd}
\end{center}
and noting that the diagonal mappings in the diagram are continuous, we obtain the continuity of the vertical mapping as a result of the fact that the $C^\infty$-topology is the initial topology induced by the $C^m$-topologies, $m\in\mathbb{Z}_{\geq 0}$.

\textbf{The $\mathbf{C^\textbf{hol}}$-case.} Since the $C^{\text{hol}}$-topology is the restriction of the $C^0$-topology, with the scalars extended to be complex and the functions restricted to be holomorphic, the analysis in $C^0$-case can be carried out verbatim to give the theorem in the holomorphic case.
\end{proof}
\begin{rmk}
    We make an important remark that Theorem \ref{thm;3.4} is valid for all regularities except for $\nu=\lip$ and $\nu=\omega$. It will be our future work to prove or disprove if this theorem is true for either of these cases.
\end{rmk}

\bibliography{main}
\bibliographystyle{plain}
\end{document}